\documentclass{amsart}   
\usepackage{graphicx, amsmath, amssymb}
\usepackage{epsf}
\usepackage{enumerate}
\DeclareGraphicsExtensions{.jpg,.pdf,.mps,.png}
\newtheorem{theorem}{Theorem}[section]

\newtheorem{claim}[theorem]{Claim}
\newtheorem{corollary}[theorem]{Corollary}
\theoremstyle{definition}

\theoremstyle{remark}

\numberwithin{equation}{section}

\begin{document}

\vspace{0.5in}

\title[Ample continua in Cartesian products of continua]%
{Ample continua in Cartesian products of continua}

\author{J. P. Boro\'nski}
\author{D. R. Prier}
\author{M. Smith}
\address[J. P. Boro\'nski]{National Supercomputing Centre IT4Innovations, Division of the University of Ostrava,
Institute for Research and Applications of Fuzzy Modeling,
30. dubna 22, 701 03 Ostrava,
Czech Republic -- and -- Faculty of Applied Mathematics,
AGH University of Science and Technology,
al. Mickiewicza 30,
30-059 Krak\'ow,
Poland}
\email{jan.boronski@osu.cz}
\address[D. R. Prier]{Mathematics Department, Gannon University,
109 University Square,
Erie, PA 16541, USA}
\email{prier001@gannon.edu}
\address[M. Smith]{Department of Mathematics and Statistics, Auburn University, AL 36849, USA}
\email{smith01@auburn.edu}


\begin{abstract}
We show that the Cartesian product of the arc and a solenoid has the fupcon property, therefore answering a question raised by Illanes. This combined with Illanes' result implies that the product of a Knaster continuum and a solenoid has the fupcon property, therefore answering a question raised by Bellamy and \L ysko in the affirmative. Finally, we show that a product of two Smith's nonmetric pseudo-arcs has the fupcon property. 
\end{abstract}
\maketitle
\section{Introduction}
The present paper is concerned with the property of having arbitrarily small open neighborhoods for continua in Cartesian products of continua; i.e. given a continuum $M\subseteq X\times Y$ we are interested if
\vspace{0.25cm}

\noindent
(*) \textit{for every open neighborhood $U$ of $M$ there exists an open and connected set $V$ such that $M\subseteq V\subseteq U$.}
\vspace{0.25cm}

\noindent
The property (*) is closely related to the property of being an ample\footnote{The notion of an ample continuum was introduced by Prajs and  Whittington in \cite{Prajs}.} continuum in the product. Recall that $M$ is \textit{ample} in $X\times Y$ provided that for each open subset $U\subseteq X\times Y$ such that $M\subseteq U$, there exists a subcontinuum $L$ of $X\times Y$ such that $M\subseteq \operatorname{int}_{X\times Y}(L)\subseteq L\subseteq U$. In fact, according to \cite{Bellamy}, the two properties are equivalent in the class of Kelley continua. Motivation for the study of ample continua comes from fact that in the hyperspace $C(X\times Y)$ of subcontinua of $X\times Y$ ample continua are the points where $C(X\times Y)$ is locally connected. In this context in \cite{Bellamy} Bellamy and \L ysko studied the \textit{fupcon}\footnote{The abbreviation \textit{fupcon} stands for \textit{full projections imply connected open neighborhoods}. It was introduced by Illanes in \cite{Illanes}.} property of Cartesian products. The product of continua $X\times Y$ has the fupcon property if whenever $M\subseteq X\times Y$ is a continuum with full projections onto coordinate spaces (i.e. $\pi_X(M)=X$ and $\pi_Y(M)=Y$) then $M$ has the property (*), and the notion naturally generalizes to Cartesian products of more than two continua. Bellamy and \L ysko showed that arbitrary Cartesian products of Knaster continua and arbitrary Cartesian products of pseudo-arcs have the fupcon property. Furthermore, the property (*) for subcontinua of such products is in fact equivalent to the property of having full projections onto all coordinate spaces. The authors also showed that the diagonal in a Cartesian square $G$ of a compact and connected topological group has the property (*) if and only if $G$ is locally connected, and therefore if $G$ is a solenoid then $G\times G$ does not have the fupcon property. Important related results on ample diagonals can be found in the recent work of Prajs \cite{Prajs2}. Motivated by the aforementioned results, Bellamy and \L ysko raised the following question.
\vspace{0.25cm}

\noindent
\textbf{Question 1.(Bellamy\&\L ysko, \cite{Bellamy})}\textit{ Let $K$ be a Knaster continuum and $S$ be a solenoid. Does $K\times S$ have the fupcon property?}
\vspace{0.25cm}

\noindent

A partial step towards a solution to the above problem was achieved by Illanes, who showed the following.
\vspace{0.25cm}

\noindent
\textbf{Theorem A. (Illanes, \cite{Illanes})}\textit{ Let $X$ be a continuum such that $X \times [0, 1]$ has the fupcon
property. Then for each Knaster continuum $K$, $X \times K$ has the fupcon
property.}
\vspace{0.25cm}

\noindent
Consequently, Question 1 was reduced to the following, potentially simpler problem.
\vspace{0.25cm}

\noindent
\textbf{Question 2.(Illanes, \cite{Illanes})}\textit{ Let $S$ be a solenoid. Does $[0,1]\times S$ have the fupcon property?}
\vspace{0.25cm}

\noindent
We answer this question in the affirmative, and in turn obtain positive answer to Question 1.
\begin{theorem}\label{main1}
Let $S$ be a solenoid. Then $[0,1]\times S$ has the fupcon property. 
\end{theorem}
\begin{theorem}\label{main2}
Let $S$ be a solenoid and $K$ be a Knaster continuum. Then $K\times S$ has the fupcon property. 
\end{theorem}
In 1985 M. Smith \cite{Smith} constructed a nonmetric pseudo-arc $\mathcal{M}$; i.e. a Hausdorff chainable, homogeneous, hereditary equivalent and hereditary indecomposable continuum. This continuum has been recently used by the first and third author to provide a new counterexample to Wood's Conjecture in the isometric theory of Banach spaces \cite{BSJMMA}. Relying on the result of Bellamy and \L ysko that products of metric pseudo-arcs have the fupcon property, we shall show that their result holds also for products of $\mathcal{M}$. 
\begin{theorem}\label{main4}
Let $\mathcal{M}$ be Smith's nonmetric pseudo-arc. Any Cartesian power of $\mathcal{M}$ has the fupcon property. 
\end{theorem} 
Earlier, Lewis showed \cite{Le} that for any 1-dimensional continuum $X$ there exists a continuum $X_P$ that admits a continuous decomposition into pseudo-arcs, and whose decomposition space is homeomorphic to $X$. Recently, Boro\'nski and Smith \cite{BoronskiSmith} extended Lewis' result to continuous curves of Smith's nonmetric pseudo-arc. In particular, given any metric 1-dimensional continuum $X$ there exists a continuum $X_\mathcal{M}$ that admits a continuous decomposition into nonmetric pseudo-arcs, and whose decomposition space is homeomorphic to $X$. $X_\mathcal{M}$ can be seen as ``$X$ of nonmetric pseudoarcs''. Here we observe that using the method of proof of Theorem \ref{main4} one obtains the following generalization.
\begin{corollary}\label{main5}
Suppose $X$ and $Y$ are metric 1-dimensional continua such that  $X_P\times Y_P$ has the fupcon property. Then $X_\mathcal{M}\times Y_\mathcal{M}$ has the fupcon property.
\end{corollary}
\section{Proofs}
\begin{proof}(of Theorem \ref{main1})
We shall assume that $S$ is the 2-adic solenoid, and give a proof for $[1,2]\times S$. For other solenoids the proof is analogous. We shall use the following inverse limit representation of $[1,2]\times S$:
$$[1,2]\times S=\lim_{\leftarrow}\{[0,1]\times \mathbb{S}_i, \operatorname{id}\times z^2_i\},$$
where $z^2_i:\mathbb{S}_{i+1}\to\mathbb{S}_i$ is the doubling map on the unit circle. For convenience we set $[1,2]\times\mathbb{S}_i=\mathbb{A}_i=\{(r,\theta):1\leq r\leq2,0\leq\theta<2\pi\}$ in polar coordinates, and $\tau_i=\operatorname{id}\times z^2_i$. Then $\tau_i:\mathbb{A}_{i+1}\to\mathbb{A}_i$ is the $2$-fold covering map given by $\tau_i(r,\theta)=(r,2\theta\mod2\pi)$ for every positive integer $i$. Let $M\subseteq [1,2]\times S$ be a continuum with full projections onto both coordinate spaces. Let $\Pi_i:[1,2]\times S\to\mathbb{A}_i$ be the projection. 
\begin{claim}\label{claim1}
$M_i=\Pi_i(M)$ is essential in $\mathbb{A}_i$ for every $i$. 
\end{claim}
\begin{proof}(of Claim \ref{claim1})
Recall that a continuum $C$ is essential in an annulus $\mathbb{A}$ if it separates the two components of the boundary. First note that $\mathbb{A}_{i+j}$ is the $2^{j}$-fold cover of $\mathbb{A}_i$ with the covering map given by 
$$\tau_{i,j}=\tau_i\circ\ldots\circ\tau_{i+j}.$$
In addition, if $\tilde{\mathbb{A}}=\{(r,\theta):1\leq r\leq2,-\infty<\theta<\infty\}$ is the universal cover, and $\phi_k:\tilde{\mathbb{A}}\to\mathbb{A}_k$ is given by $\phi_k(r,\theta)=(r,2^k\theta\mod2\pi)$ then $\phi_i=\tau_{i,j}\circ\phi_{i+j}$. 
\begin{figure}[ht]
	\centering
		\includegraphics[width=0.49\textwidth]{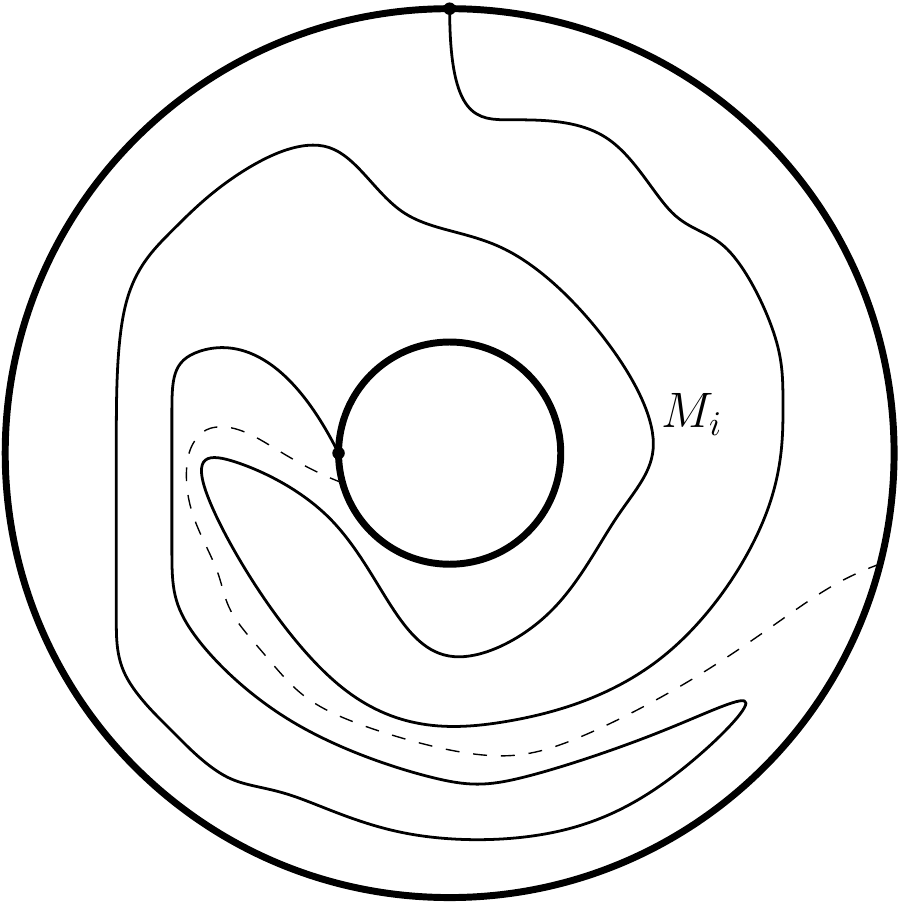}
	\caption{Proof of Claim \ref{claim1}: an inessential continuum $M_i$ with full projections in $[1,2]\times \mathbb{S}_i$}
	\label{fig:fupconpic}
\end{figure}
By contradiction suppose $M_i$ is inessential in $\mathbb{A}_i$. Then $M_i$ is contained in a closed disk $D_i$. Since $D_i$ is simply connected, for each positive integer $i$ any component of $\tau^{-1}_{i}(M_{i})$ is a homeomorphic copy of $M_i$. In particular $M_{i+j}$ is homeomorphic to $M_i$ for each $j$. In addition,
$$\lim_{j\to\infty}\operatorname{diam}(M_{i+j})=0.$$
Indeed, since any component $\tilde M$ of $\phi^{-1}_i(M_i)$ in the universal cover $\tilde{\mathbb{A}}$ is bounded, one can take $j$ large enough so that the projection $N_{i+j}=\phi_{i+j}(\tilde{M})$ from $\tilde{\mathbb{A}}$ onto $\mathbb{A}_{i+j}$ is as small as desired. In particular, it is true when $N_{i+j}=M_{i+j}$. Therefore, there exists a $j_o$ such that the projection of $M_{i+j}\subseteq\mathbb{A}_{i+j}$ onto $\mathbb{S}_{i+j}$ is a proper subset of $\mathbb{S}_{i+j}$. This implies that $M$ does not have a full projection onto $S$, resulting in a contradiction and completing the proof of Claim \ref{claim1}. 
\end{proof}
\begin{claim}\label{claim2}
$\tau^{-1}_i(M_i)=M_{i+1}$ for every $i$; i.e. $\tau^{-1}_i(M_i)$ is connected for each $i$. 
\end{claim}
\begin{proof}(of Claim \ref{claim2})
We use a similar argument to that of Example 1 in \cite{Heath}. By contradiction, suppose $\tau^{-1}_i(M_i)$ is disconnected. Without loss of generality let us assume that $i=1$. Then there are two components $M_2$ and $N_2$ of $\tau^{-1}_1(M_1)$, and each of them maps onto $M_1$. They are homeomorphic, since the map $\sigma:\mathbb{A}_i\to \mathbb{A}_i$ given by $\sigma(r,\theta)=\left(r,(\theta+2\pi)\mod 2\pi\right)$ is a homeomorphism with $\sigma(M_2)=N_2$. By Claim \ref{claim1} $M_2$ is essential. Since $M_2$ has a full projection onto $[1,2]$ it must connect the two boundary circles and so there is a point $c\in M_2\cap N_2$. This contradiction implies that $M_2=N_2$ and completes the proof of Claim \ref{claim2}.  
\end{proof}
To finish the proof of Theorem \ref{main1} let $W$ be an open neighborhood of $M$. Note that since $\mathbb{A}_1$ is locally connected, $M_1$ has arbitrarily small connected open neighborhoods. So if $W_1=\Pi_1(W)$ then there exists a connected open neighborhood $U_1$ such that $M_1\subseteq U_1\subseteq W_1$. Reasoning as above in Claim \ref{claim2}, we deduce that $U_2=\tau^{-1}_1(U_1)$ is connected, and then proceeding by induction that $U_{i+1}=\tau^{-1}_i(U_i)$ we obtain that $U_{i+1}$ is connected for each $i$. Consequently $U=(\tau^{-1}_i(U_i):i=1,2,\ldots)$ is an open and connected set such that $M\subseteq U\subseteq W$ and Theorem \ref{main1} is proved. 
\end{proof}

\begin{proof}(of Theorem \ref{main2})
This follows from Illanes' Theorem A, since by Theorem \ref{main1} $[0,1]\times S$ has the fupcon property. 
\end{proof}

\begin{proof}(of Theorem \ref{main4}) For simplicity of notation we prove it for the product of two Smith's pseudo-arcs. The general case is similar thanks to the result of Bellamy and Lysko for arbitrary products of pseudo-arcs. 

We consider $\mathcal{M}$ as the following long inverse limit
$$\mathcal{M}=\lim_{\leftarrow}\{P_\alpha,p^{\beta}_{\alpha}:\alpha<\beta<\omega_1\},$$
where each $P_\alpha$ is a metric pseudo-arc, and $p^{\beta}_{\alpha}:P_\beta\to P_\alpha$ is an open, closed and monotone map, such that $(p^{\beta}_{\alpha})^{-1}(x)$ is a pseudo-arc contained in $P_\beta$ for each $x\in P_\alpha$. Consider the Cartesian square of $\mathcal{M}$ as the following inverse limit.
$$\mathcal{M}\times \mathcal{M}'=\lim_{\leftarrow}\{P_\alpha\times P'_\alpha,p^{\beta}_{\alpha}\times q^{\beta}_{\alpha}:\alpha<\beta<\omega_1\}.$$
Let $\Gamma_\alpha:\mathcal{M}\times\mathcal{M}\to P_\alpha\times P'_\alpha$ be given by
$$\Gamma_\alpha((\{x_\alpha\}_{\alpha<\omega_1},\{y_\alpha\}_{\alpha<\omega_1}))=(x_\alpha,y_\alpha).$$ 
Note that $\Gamma_\alpha$ is monotone (i.e. pre-images of points are connected), open and closed for each $\alpha<\omega_1$. Let $M\subseteq \mathcal{M}\times\mathcal{M}$ be a continuum with full projections onto both coordinate spaces, and $W$ be an open set around $M$. Then the projection of $W$ onto the square of $\alpha$th coordinate spaces $W_\alpha=\Gamma_\alpha(W)$ is an open set around the continuum $M_\alpha=\Gamma_\alpha(M)$. Since $M_\alpha$ has full projections onto both coordinate spaces $P_\alpha$ and $P'_\alpha$, by Theorem 4.4 in \cite{Bellamy}, there exists an open and connected set $V_\alpha$ such that $M_\alpha\subseteq V_\alpha\subseteq W_\alpha$. By Theorem 6.1.29. in \cite{engelking}, p.358, it follows that $V=\Gamma_{\alpha}^{-1}(V_\alpha)$ is an open and connected set, such that $M\subseteq V\subseteq W$. This completes the proof.   
\end{proof}
The proof of Corollary \ref{main5} is analogous to the one of Theorem \ref{main4}, and is left to the reader. We conclude with the following questions.
\vspace{0.25cm}

\noindent
\textbf{Question 3.(Bellamy\&\L ysko, \cite{Bellamy})}\textit{ Does the product of two nonhomeomorphic solenoids have the fupcon property?}
\vspace{0.25cm}

\noindent
\textbf{Question 4.(Illanes \cite{Illanes})}\textit{Let $X$ and $Y$ be chainable Kelley continua. Does $X
\times Y$ have the fupcon property?}
\vspace{0.25cm}

\noindent
\textbf{Question 5.}\textit{ Suppose $X$ and $Y$ are 1-dimensional continua such that $X\times Y$ has the fupcon property. Does the product $X_P\times Y_P$ have the fupcon property?}
\vspace{0.25cm}

\noindent
\textbf{Question 6.}\textit{ Does the product of $[0,1]$ and pseudo-circle have the fupcon property?}
\vspace{0.25cm}

\noindent
\textbf{Question 7.}\textit{ Does the product of a pseudo-arc and pseudo-circle have the fupcon property?}
\vspace{0.25cm}

\noindent
\textbf{Question 8.}\textit{ Does the product of two pseudo-circles have the fupcon property?}

\section*{Acknowledgements}
The first author is grateful to J. Prajs for drawing his attention to the topic of this paper and some informative conversations during the 32nd Summer Conference on Topology and its Applications in June 2017 at the University of Dayton. It was during that conference and the following two weeks when an important part of this collaboration was carried out. It was made possible thanks to the support from the Moravian-Silesian Region of the Czech Republic by MSK grant 01211/2016/RRC ``Strengthening international cooperation in science, research and education''. This work was also supported by the NPU II project LQ1602 IT4Innovations excellence in science.


\begin{thebibliography}{99}
\bibitem{Bellamy}{\sc Bellamy, D. P.; \L ysko, J. M.}{\em Connected open neighborhoods of subcontinua of product continua with indecomposable factors.} \textbf{Topology Proc.}, 44 (2014), 223--231. 
\bibitem{BSJMMA} {\sc Boro\'nski, J.P.; Smith, M.} {\em On the conjecture of Wood and projective homogeneity}  arXiv:1607.04105
\bibitem{BoronskiSmith} {\sc Boro\'nski, J.P.; Smith, M.} {\em Continuous curves of nonmetric pseudo-arcs and semi-conjugacies to interval maps} arXiv:1701.01862
\bibitem{Heath}{\sc Heath, J. W.}, {\em Weakly confluent, $2$-to-$1$ maps on hereditarily indecomposable continua.}
\textbf{Proc. Amer. Math. Soc.} 117 (1993), 569--573. 
\bibitem{engelking} {\sc Engelking, R.} {\em General topology.} 2nd Edition, Sigma Series in Pure Mathematics, 6. Heldermann Verlag, Berlin, 1989.
 \bibitem{Illanes}{\sc Illanes, A.} {\em Connected open neighborhoods in products.} \textbf{Acta Math. Hungar.} 148 (2016), 73--82.
\bibitem{Le} {\sc Lewis, W.}, {\em Continuous curves of pseudo-arcs.} \textbf{Houston J. Math. }11 (1985), pp. 91--99.
\bibitem{Prajs}{\sc Prajs, J.; Whittington, K.}{\em Filament sets, aposyndesis, and the decomposition theorem of Jones.} \textbf{Trans. Amer. Math. Soc.} 359 (2007), no. 12, 5991-6000.
\bibitem{Prajs2}{\sc Prajs, J.}{\em Isometrically homogeneous and topologically homogeneous continua.}
\textbf{Indiana Univ. Math. J.} 65 (2016), 1289--1306.
\bibitem{Smith} {\sc Smith, M.} {\em On nonmetric pseudo-arcs}, \textbf{Topology Proc.} 10 (1985), no. 2, 385--397. 
\end{thebibliography}
\end{document}